\documentclass{amsart}

\usepackage{amssymb}
\usepackage{amstext}
\usepackage{amsmath}
\usepackage{amscd}
\usepackage{amsthm}
\usepackage{amsfonts}
\usepackage{enumerate}
\usepackage{graphicx}
\usepackage{color}
\usepackage{here} 
\usepackage{bm} 
\usepackage{mathrsfs}
\usepackage{mathtools}
\usepackage{latexsym}
\usepackage{booktabs}
\usepackage{fancyhdr}
\usepackage{subcaption}
\usepackage{tikz}
\usepackage{tikz-cd}
\usetikzlibrary{arrows}
\usetikzlibrary{decorations.markings}
\usetikzlibrary{patterns,decorations.pathreplacing}
\usepackage{mathtools}
\usepackage{pdfpages}
\usepackage[section]{algorithm}
\usepackage{algpseudocodex}

\usepackage{hyperref}
\hypersetup{
  colorlinks=true,
  citecolor=red,
  linkcolor=blue
}
\theoremstyle{plain}
\newtheorem{theorem}{Theorem}[section]
\newtheorem{lemma}[theorem]{Lemma}
\newtheorem{proposition}[theorem]{Proposition}
\newtheorem{corollary}[theorem]{Corollary}

\theoremstyle{definition}

\newtheorem{definition}[theorem]{Definition}
\newtheorem{example}[theorem]{Example}

\theoremstyle{remark}

\newcommand{\groebner}{Gr\"{o}bner }

\newcommand{\intnonneg}{\mathbb{Z}_{\geq 0}}

\newcommand{\calA}{\mathcal{A}}

\newcommand{\calC}{\mathcal{C}}

\newcommand{\calF}{\mathcal{F}}

\newcommand{\calK}{\mathcal{K}}

\newcommand{\calM}{\mathcal{M}}

\newcommand{\ZZ}{\mathbb{Z}}

\newcommand{\kk}{\Bbbk}

\newcommand{\scrM}{\mathscr{M}}

\newcommand{\ab}{\mathbf{a}}

\newcommand{\eb}{\mathbf{e}}

\newcommand{\tb}{\mathbf{t}}

\newcommand{\fkp}{\mathfrak{p}}

\newcommand{\fkS}{\mathfrak{S}}

\DeclareMathOperator{\codim}{codim}
\DeclareMathOperator{\length}{length}
\DeclareMathOperator{\GL}{GL}
\DeclareMathOperator{\borel}{B}

\DeclareMathOperator{\ini}{in}
\DeclareMathOperator{\ind}{ind}
\DeclareMathOperator{\gin}{gin}
\DeclareMathOperator{\HS}{HS}
\DeclareMathOperator{\HF}{HF}

\title{Standard multigraded Hibi rings and Cartwright-Sturmfels ideals}
\date{}

\author[K. Matsushita]{Koji Matsushita}
\address{Graduate School of Mathematical Sciences, University of Tokyo, Komaba, Meguro-ku, Tokyo 153-8914, Japan}
\email{koji-matsushita@g.ecc.u-tokyo.ac.jp}

\author[K. Tani]{Koichiro Tani}
\address{Department of Pure and Applied Mathematics, Graduate School of Information
Science and Technology, Osaka University, Osaka, Japan}
\email{tani-k@ist.osaka-u.ac.jp}


\begin{document}
\begin{abstract}
    In this paper, we introduce standard multigradings on Hibi rings, which are algebras arising from posets. 
    We show that any standard multigrading on a Hibi ring that makes its defining ideal (called the Hibi ideal) homogeneous is induced by a chain of the underlying poset.
    After that, we calculate the multigraded Hilbert series of Hibi rings by generalizing the theory of $P$-partition and we compute the multidegree polynomials of Hibi rings.
    Furthermore, we characterize Hibi ideals that are Cartwright-Sturmfels ideals.
\end{abstract}
\maketitle

\section{Introduction}\label{section:Introduction}
We assume that all posets and lattices appearing in this paper are finite.
Let $L$ be a lattice and let $\kk$ be a field.
Let $S_{L} = \kk[x_{\alpha} : \alpha \in L]$ be the polynomial ring over $\kk$ and $I_{L}$ be the ideal of $S_{L}$ generated by
\begin{equation*}
  \calF_{L} \coloneqq \{f_{\alpha, \beta} : \alpha, \beta \in L \text{ are incomparable.}\},
\end{equation*}
where $f_{\alpha, \beta} \coloneqq x_{\alpha}x_{\beta}-x_{\alpha \wedge \beta}x_{\alpha \vee \beta}$
for incomparable elements $\alpha, \beta \in L$.
A \textit{compatible monomial order} is a monomial order $\preceq$ on $S_{L}$ such that
$\ini_{\preceq}(f_{\alpha, \beta}) = x_{\alpha}x_{\beta}$
for all incomparable elements $\alpha, \beta \in L$.
It is known that the following are equivalent (\cite{HibiRing}, see also \cite[Chapter~6]{HHOBinomialIdeals}):
\begin{enumerate}
  \item $I_{L}$ is a prime ideal;
  \item $L$ is a distributive lattice;
  \item $\calF_{L}$ is a \groebner basis of $I_{L}$ with respect to a compatible monomial order.
\end{enumerate}
When the equivalent conditions hold, the quotient ring $S_{L}/I_{L}$ is called the \textit{Hibi ring} of $L$, introduced in \cite{HibiRing}.
In this paper, we call $I_{L}$ the \textit{Hibi ideal} of $L$ if the equivalent condition holds.
For more details on Hibi rings, see~\cite[Chapter 6]{HHOBinomialIdeals}.

To study Hibi rings, we use a one-to-one correspondence between posets and distributive lattices.
Let $P$ be a poset equipped with a partial order $\leq_P$. A \textit{poset ideal} of $P$ is a subset $\alpha$ of $P$ satisfying that, whenever $a \in \alpha$ and $b \in P$ with $b \leq_P a$, one has $b \in \alpha$.
Let $L(P)$ be the set of poset ideals of $P$.
Since the empty set and $P$ itself are elements of $L(P)$, and since $\alpha \cap \beta$ and $\alpha \cup \beta$ belong to $L(P)$ for all
$\alpha,\beta \in L(P)$, the set $L(P)$ forms a distributive lattice ordered by inclusion, 
where the join and meet operations correspond to the union and intersection, respectively.
Conversely, the poset $P$ can be recovered from $L(P)$ as follows:
We say that an element $\alpha \in L(P)\setminus \{\emptyset\}$ is \textit{join-irreducible} if, whenever $\alpha=\beta\cup\gamma$ with $\beta,\gamma\in L(P)$, one has either $\alpha=\beta$ or $\alpha=\gamma$.
Then the subposet of $L(P)$ consisting of all join-irreducible elements coincides with $P$.
This is well known as Birkhoff's representation theorem \cite{Birkhoff1967}.
We denote the Hibi ring arising from $P$ by $\kk[P]$, that is, $\kk[P]:=S_{L(P)}/I_{L(P)}$.

Usually, the Hibi ring $S_L/I_L$ of a distributive lattice $L$ is regarded as a standard $\ZZ$-graded ring by setting $\deg(x_\alpha)=1$.
In this paper, we introduce a standard multigrading on the Hibi ring and study its properties.

\medskip

First, we discuss how to define standard multigradings on Hibi rings, i.e., how to define a multigrading on $S_{L(P)}$ for a poset $P$ such that $I_{L(P)}$ is homogeneous.
In this paper, we use the term \textit{homogeneous} in the multigraded sense.
We give a multigrading on $S_{L(P)}$ induced by a chain of $P$
(Proposition~\ref{prop:multigrading is compatible}) and show that any multigrading that makes $I_{L(P)}$ a homogeneous ideal coincides with ours (Theorem~\ref{theorem:multigraded with a chain is canonical}).


After we introduce multigradings, we study three topics about multigraded Hibi rings; multigraded Hilbert series, multidegree polynomials and multigraded generic initial ideals.
We determine the multigraded Hilbert series of multigraded Hibi rings.
In the case where $S_{L(P)}$ is considered as the standard $\ZZ$-graded ring, 
the Hilbert series of $\kk[P]$ is given in \cite[Proposition~2.3]{Hibi1991Hilbert}, where it is computed using the theory of $P$-partition~\cite{StanleyP-partitions}.
We extend this method to the multigrading case (Theorem~\ref{theorem:Hilbert series of multigraded Hibi rings} and Corollary~\ref{corollary:Hilbert series}).

Moreover, using these results, we calculate the multidegree polynomials~\cite[Section 8.5]{MillerSturmfels} of multigraded Hibi rings (Proposition~\ref{proposition:multidegree of Hibi rings}).
The notion of multidegree is a generalization of the notion of degree of graded rings and modules to a multigraded settings. One can compute the multidegree polynomial from the $\calK$-polynomial.
Multidegree polynomials have applications, for example, in algebraic geometry and convex geometry (see, e.g.,~\cite{CaminataCidRuizConca},~\cite{TrungVermaMixedVolume}).


Finally, we discuss Cartwright-Sturmfels Hibi ideals.
We say that a homogeneous ideal is \textit{Cartwright-Sturmfels} if there exists a squarefree
multigraded Borel-fixed monomial ideal such that its multigraded Hilbert series is
the same as that of the homogeneous ideal. It was shown in~\cite[Proposition 2.6]{ConcaDeNegriGorlaUGBandCS}
that, if the ground field is infinite, a homogeneous ideal is Cartwright-Sturmfels
if and only if a multigraded generic initial ideal of it is squarefree.
In many classes of ideals studied in combinatorial commutative algebra, the relationship
between these ideals and Cartwright-Sturmfels ideals was studied. Examples include
determinantal ideals~\cite{CSHilbertScheme, ConcaDeNegriGorlaUGBforMM},
Schubert determinantal ideals~\cite{ConcaDeNegriGorlaRadGin},
binomial edge ideals~\cite{ConcaDeNegriGorlaMulGinofDet,ConcaDeNegriGorlaCSassociGraph},
and so on.
In this paper, we determine when Hibi ideals are Cartwright-Sturmfels ideals (Theorem~\ref{theorem:characterization of CS Hibi ideal}).

\bigskip

This paper is organized as follows. In Section~\ref{section:Multigrading with a chain},
we introduce standard multigradings on Hibi ideals and study these multigradings.
In Section~\ref{section:Hilbert series of multigraded Hibi rings},
we review the theory of $P$-partition and refine it in order to compute the
multigraded Hilbert series of Hibi rings under our multigradings and compute them.
In Section~\ref{section:Multidegree polynomials of multigraded Hibi rings},
we compute multidegree polynomials of multigraded Hibi rings in terms of chains of distributive lattices.
In Section~\ref{section:When do Hibi ideals become Cartwright-Sturmfels},
we review the multigraded generic initial ideals and Cartwright-Sturmfels ideals,
and we characterize Hibi ideals that are Cartwright-Sturmfels ideals
under our multigradings.

\subsection*{Acknowledgements}
The authors would like to thank Akihiro Higashitani for giving helpful discussions.
The first author was partially supported by Grant-in-Aid for JSPS Fellows Grant JP25KJ0047.

\section{Multigrading with a chain}\label{section:Multigrading with a chain}
In this section, we introduce and study multigradings on Hibi rings.
For integers $a$ and $b$ with $a\le b$, let $[a,b]:=\{c\in \ZZ : a\le c \le b\}$.
Throughout this section, let $P:=[1,n]$ be a poset equipped with $\leq_P$.

First, for a nonnegative integer $m$, we define a standard $\ZZ^{m+1}$-multigrading on $S_{L(P)}$, that is, each variable of $S_{L(P)}$ has degree $\eb_i$ for some $i\in [0,m]$ where $\eb_i$ stands for the $i$th unit vector of $\ZZ^{m+1}$.
Let $C=\{c_1<_P\cdots<_Pc_\ell\}$ be a chain of $P$ and $C_i:=\{c_1,\ldots,c_i\}$ for $i=0,1,\ldots,\ell$ where we let $C_0:=\emptyset$.
Moreover, let $\fkp(C):=\{C_0,C_1,\ldots,C_\ell\}$ and let $f_C : \fkp(C) \to [0,m]$ be a map.
Note that for any poset ideal $\alpha$ of $P$, the set $\alpha\cap C$ belongs to $\fkp(C)$.
We define the standard $\ZZ^{m+1}$-multigrading on $S_{L(P)}$ by setting $\deg_{f_C}(x_\alpha):=\eb_{f_C(\alpha\cap C)}$.
In particular, when $m=\ell$ and $f_C(x_\alpha)=|\alpha\cap C|$ for any $\alpha\in L(P)$, we simply write $\deg_C$ instead of $\deg_{f_C}$.


\begin{proposition}\label{prop:multigrading is compatible}
  Work with the same notation as above.
    Then $I_{L(P)}$ is homogeneous.
\end{proposition}
\begin{proof}
  It is enough to show that for $\alpha,\beta\in L(P)$, the binomial $f_{\alpha,\beta}$ is a homogeneous element of $S_{L(P)}$.
    Since $\alpha$ and $\beta$ are poset ideals of $P$, we can write $\alpha\cap C=C_i$ and $\beta\cap C=C_j$ for some $i$ and $j$.
    We may assume that $i\le j$.
    Then we can see that $(\alpha\cup\beta)\cap C=C_j$ and $(\alpha\cap\beta)\cap C=C_i$.
    Therefore, we have 
    \[
    \deg_{f_C}(x_\alpha x_\beta)=\deg_{f_C}(x_{\alpha\cup\beta}x_{\alpha\cap\beta})=\eb_{f_C(C_i)}+\eb_{f_C(C_j)},
    \]
    and hence $f_{\alpha,\beta}$ is a homogeneous element.
\end{proof}

\begin{example}\label{example:multigradings of N-poset}
  Let $P$ be a poset depicted in Figure~\ref{fig:N-poset}.
  Then, the distributive lattice $L(P)$ associated with $P$ is depicted as in Figure~\ref{fig:distributive lattice of N-poset}.
  
  If $S_{L(P)}$ is multigraded by the chain $C_1 := \{2,3\}$, then
  \[
  \deg_{C_1}(x_{\alpha}) =
  \left\{
  \begin{aligned}
      &\eb_0 \quad (\alpha = \emptyset, \{1\}), \\
      &\eb_1 \quad (\alpha = \{2\}, \{1,2\}, \{2,4\}, \{1,2,4\}), \\
      &\eb_2 \quad (\alpha = \{1,2,3\}, \{1,2,3,4\}). \\
  \end{aligned}
  \right.
  \]
  Therefore, the elements of $L(P)$ are grouped as in Figure~\ref{fig:multigraded with 2,3}.

  If $S_{L(P)}$ is multigraded by the chain $C_2 := \{2,4\}$, then
  \[
  \deg_{C_2}(x_{\alpha}) =
  \left\{
  \begin{aligned}
      &\eb_0 \quad (\alpha = \emptyset, \{1\}), \\
      &\eb_1 \quad (\alpha = \{2\}, \{1,2\}, \{1,2,3\}), \\
      &\eb_2 \quad (\alpha = \{2,4\}, \{1,2,4\}, \{1,2,3,4\}). \\
  \end{aligned}
  \right.
  \]
  Therefore, the elements of $L(P)$ are grouped as in Figure~\ref{fig:multigraded with 2,4}.

  \begin{figure}[ht]
    \begin{minipage}{0.48\columnwidth}
      \centering
      {\scalebox{0.5}{
        \begin{tikzpicture}[line width=0.05cm]
          \coordinate (N1) at (0,0); 
          \coordinate (N2) at (2,0); 
          \coordinate (N3) at (0,2); 
          \coordinate (N4) at (2,2);
          \draw  (N1)--(N3);
          \draw  (N2)--(N3);
          \draw  (N2)--(N4);
          \draw [line width=0.05cm, fill=white] (N1) circle [radius=0.15] node [left=2pt] {\Huge $1$};
          \draw [line width=0.05cm, fill=white] (N2) circle [radius=0.15] node [right=2pt] {\Huge $2$};
          \draw [line width=0.05cm, fill=white] (N3) circle [radius=0.15] node [left=2pt] {\Huge $3$};
          \draw [line width=0.05cm, fill=white] (N4) circle [radius=0.15] node [right=2pt] {\Huge $4$};
        \end{tikzpicture}
      }}
      \caption{The poset $P$}
      \label{fig:N-poset}
    \end{minipage}
    \begin{minipage}{0.48\columnwidth}
      \centering
      {\scalebox{0.4}{
        \begin{tikzpicture}[line width=0.05cm]
          \coordinate (a1) at (0, 0);
          \coordinate (a2) at (-2, 2);
          \coordinate (a3) at (2, 2);
          \coordinate (a5) at (0, 4);
          \coordinate (a6) at (4, 4);
          \coordinate (a7) at (-2, 6);
          \coordinate (a8) at (2, 6);
          \coordinate (a9) at (0, 8);
          
          \draw (a1) -- (a2);
          \draw (a1) -- (a3);
          \draw (a2) -- (a5);
          \draw (a3) -- (a5);
          \draw (a3) -- (a6);
          \draw (a5) -- (a7);
          \draw (a5) -- (a8);
          \draw (a6) -- (a8);
          \draw (a7) -- (a9);
          \draw (a8) -- (a9);
      
          \draw [line width=0.05cm, fill=white] (a1) circle [radius=0.15] node [below=2pt] {\Huge $\emptyset$};
          \draw [line width=0.05cm, fill=white] (a2) circle [radius=0.15] node [below left=1pt] {\Huge $\{1\}$};
          \draw [line width=0.05cm, fill=white] (a3) circle [radius=0.15] node [below right=1pt] {\Huge $\{2\}$};
          \draw [line width=0.05cm, fill=white] (a5) circle [radius=0.15] node [left=8pt] {\Huge $\{1, 2\}$};
          \draw [line width=0.05cm, fill=white] (a6) circle [radius=0.15] node [right=2pt] {\Huge $\{2, 4\}$};
          \draw [line width=0.05cm, fill=white] (a7) circle [radius=0.15] node [above left=1pt] {\Huge $\{1, 2, 3\}$};
          \draw [line width=0.05cm, fill=white] (a8) circle [radius=0.15] node [above right=2pt] {\Huge $\{1, 2, 4\}$};
          \draw [line width=0.05cm, fill=white] (a9) circle [radius=0.15] node [above=2pt] {\Huge $\{1, 2, 3, 4\}$};
      \end{tikzpicture}
        }}
        \caption{The distributive lattice $L(P)$}
        \label{fig:distributive lattice of N-poset}
    \end{minipage}
  \end{figure}
  \begin{figure}[ht]
    \begin{minipage}{0.48\columnwidth}
      \centering
      {\scalebox{0.4}{
        \begin{tikzpicture}[line width=0.05cm]
          \coordinate (a1) at (0, 0);
          \coordinate (a2) at (-2, 2);
          \coordinate (a3) at (2, 2);
          \coordinate (a5) at (0, 4);
          \coordinate (a6) at (4, 4);
          \coordinate (a7) at (-2, 6);
          \coordinate (a8) at (2, 6);
          \coordinate (a9) at (0, 8);
          
          \draw (a1) -- (a2);
          \draw (a1) -- (a3);
          \draw (a2) -- (a5);
          \draw (a3) -- (a5);
          \draw (a3) -- (a6);
          \draw (a5) -- (a7);
          \draw (a5) -- (a8);
          \draw (a6) -- (a8);
          \draw (a7) -- (a9);
          \draw (a8) -- (a9);
      
          \draw [line width=0.05cm, fill=white] (a1) circle [radius=0.15];
          \draw [line width=0.05cm, fill=white] (a2) circle [radius=0.15];
          \draw [line width=0.05cm, fill=white] (a3) circle [radius=0.15];
          \draw [line width=0.05cm, fill=white] (a5) circle [radius=0.15];
          \draw [line width=0.05cm, fill=white] (a6) circle [radius=0.15];
          \draw [line width=0.05cm, fill=white] (a7) circle [radius=0.15];
          \draw [line width=0.05cm, fill=white] (a8) circle [radius=0.15];
          \draw [line width=0.05cm, fill=white] (a9) circle [radius=0.15];

          \draw [very thick, rotate around={45:(-1,1)}] (-1,1) ellipse (1 and 2);
          \draw [very thick, rotate around={0:(2,4)}] (2,4) ellipse (2.5 and 2.5);
          \draw [very thick, rotate around={-45:(-1,7)}] (-1,7) ellipse (1 and 2);
          
          \node at (-2,0) [font=\Huge] {$\eb_0$};
          \node at (5,4) [font=\Huge] {$\eb_1$};
          \node at (-2,8) [font=\Huge] {$\eb_2$};
      \end{tikzpicture}
      }}
      \caption{Multigrading by $C_1$}
      \label{fig:multigraded with 2,3}
    \end{minipage}
    \begin{minipage}{0.48\columnwidth}
      \centering
      {\scalebox{0.4}{
        \begin{tikzpicture}[line width=0.05cm]
          \coordinate (a1) at (0, 0);
          \coordinate (a2) at (-2, 2);
          \coordinate (a3) at (2, 2);
          \coordinate (a5) at (0, 4);
          \coordinate (a6) at (4, 4);
          \coordinate (a7) at (-2, 6);
          \coordinate (a8) at (2, 6);
          \coordinate (a9) at (0, 8);
          
          \draw (a1) -- (a2);
          \draw (a1) -- (a3);
          \draw (a2) -- (a5);
          \draw (a3) -- (a5);
          \draw (a3) -- (a6);
          \draw (a5) -- (a7);
          \draw (a5) -- (a8);
          \draw (a6) -- (a8);
          \draw (a7) -- (a9);
          \draw (a8) -- (a9);
      
          \draw [line width=0.05cm, fill=white] (a1) circle [radius=0.15];
          \draw [line width=0.05cm, fill=white] (a2) circle [radius=0.15];
          \draw [line width=0.05cm, fill=white] (a3) circle [radius=0.15];
          \draw [line width=0.05cm, fill=white] (a5) circle [radius=0.15];
          \draw [line width=0.05cm, fill=white] (a6) circle [radius=0.15];
          \draw [line width=0.05cm, fill=white] (a7) circle [radius=0.15];
          \draw [line width=0.05cm, fill=white] (a8) circle [radius=0.15];
          \draw [line width=0.05cm, fill=white] (a9) circle [radius=0.15];

          \draw [very thick, rotate around={45:(-1,1)}] (-1,1) ellipse (1 and 2);
          \draw [very thick, rotate around={45:(0,5)}] (a5) ellipse (1 and 3.5);
          \draw [very thick, rotate around={45:(2,6)}] (a8) ellipse (1 and 3.5);

          \node at (1,-1) [font=\Huge] {$\eb_0$};
          \node at (3,1) [font=\Huge] {$\eb_1$};
          \node at (5,3) [font=\Huge] {$\eb_2$};
      \end{tikzpicture}
        }}
        \caption{Multigrading by $C_2$}
        \label{fig:multigraded with 2,4}
    \end{minipage}
  \end{figure}
\end{example}

The following theorem says that our grading is essential for $I_{L(P)}$ to be homogeneous:

\begin{theorem}\label{theorem:multigraded with a chain is canonical}
 Suppose that $S_{L(P)}$ is standard $\ZZ^{m+1}$-multigraded with a degree map $\deg$ such that $I_{L(P)}$ is homogeneous.
Then there exist a chain $C$ of $P$ and a map $f_C : \fkp(C) \to [0,m]$ with $\deg=\deg_{f_C}$.
\end{theorem}
\begin{proof}
  For a nonnegative integer $k$, we set $L_k(P):=\{\alpha\in L(P) : |\alpha|=k\}$.
    Since $S_{L(P)}$ is standard multigraded, we have $\deg(x_\emptyset)=\eb_{i_0}$ for some $i_0\in [0,m]$.
    If $\deg(x_\alpha)=\eb_{i_0}$ for any $\alpha\in L(P)$, then by setting $C=\emptyset$ and defining $f_C:\fkp(C)\to [0,m]$ by $f(\emptyset)=i_0$, we have $\deg=\deg_{f_C}$.
    Otherwise, let $j_1$ be the smallest integer satisfying $M_{j_1}:=\{\alpha\in L_{j_1}(P) : \deg(x_\alpha)\neq \eb_{i_0}\}\neq \emptyset$.
    
    We show that $|M_{j_1}|=1$ and the unique element $\gamma_1$ in $M_{j_1}$ is join-irreducible.
    If $|M_{j_1}|\ge 2$, then for $\alpha,\beta\in M_{j_1}$ with $\alpha\neq\beta$, we have $\alpha\cap \beta\in L_k(P)$ for some $k<j_1$, and hence $\deg(x_{\alpha\cap \beta})=\eb_{i_0}$.
    Since $\deg(x_\alpha)\neq \eb_{i_0}$ and $\deg(x_\beta)\neq\eb_{i_0}$, the binomial $f_{\alpha,\beta}$ is not a homogeneous element, a contradiction because $I_{L(P)}$ is homogeneous.
    If $\gamma_1$ is not join-irreducible, then there exist $\alpha,\beta\in L(P)$ with $\gamma_1=\alpha\cup \beta$, $\alpha\neq \gamma_1$ and $\beta\neq \gamma_1$.
    In this case, we can see that $\alpha\in L_{k_1}(P)$ and $\beta\in L_{k_2}(P)$ for some $k_1,k_2<j_1$, so we have $\deg(x_\alpha)=\deg(x_\beta)=\eb_{i_0}$. 
    This implies that $f_{\alpha,\beta}$ is not a homogeneous element by the same argument as above, a contradiction. 

    Let $i_1$ be the integer satisfying $\deg(x_{\gamma_1})=\eb_{i_1}$ and consider the subposet $N_1:=\{\alpha \in L(P) : \gamma_1\subset \alpha\}$ of $L(P)$, which coincides with the distributive lattice $L(P\setminus \gamma_1)$.
    Then we have $\deg(x_\alpha)=\eb_{i_0}$ for any $\alpha \in L(P)\setminus N_1$.
    Indeed, $\alpha$ and $\gamma_1$ are incomparable in $L(P)$, and $\alpha\cap \gamma_1 \in L_k(P)$ for some $k<j_1$.
    Since $\deg(x_{\gamma_1})=\eb_{i_1}$ and $\deg(x_{\alpha\cap \gamma_1})=\eb_{i_0}$, the degree of $x_\alpha$ must be $\eb_{i_0}$.
    
    We define $j_2$ as the smallest integer satisfying $M_{j_2}:=\{\alpha\in L_{j_2}(P)\cap N_1 : \deg(x_\alpha)\neq \eb_{i_1}\}\neq \emptyset$ if such an integer exists.
    By a similar argument as above, we can see that $|M_{j_2}|=1$ and the unique element $\gamma_2$ in $M_{j_2}$ is a join-irreducible element in $N_1$.
    Actually, $\gamma_2$ is also join-irreducible in $L(P)$.
    Suppose that there exist $\alpha,\beta\in L(P)$ with $\gamma_2=\alpha\cup \beta$, $\alpha\neq \gamma_2$ and $\beta\neq \gamma_2$.
    Since $\gamma_1$ is join-irreducible and $\gamma_1\subset \gamma_2$, we have $\gamma_1\subset \alpha$ or $\gamma_1\subset \beta$.
    We may assume that $\gamma_1\subset \alpha$.
    Since $\gamma_2$ is join-irreducible in $N_1$, we obtain $\beta\in L(P)\setminus N_1$, and hence $\alpha\cap \beta\in L(P)\setminus N_1$.
    We have $\deg(x_\alpha)=\eb_{i_1}$ and $\deg(x_\beta)=\deg(x_{\alpha\cap\beta})=\eb_{i_0}$, so $\deg(x_{\alpha\cup \beta})=\eb_{i_1}$, a contradiction to $\alpha\cup\beta=\gamma_2\in M_{j_2}$.
    
    We continue to define $j_s$, $M_{j_s}$, $\gamma_{j_s}$, $N_s$, and $i_s$ for $s\ge 2$ by the same procedure.
    That is, let $j_s$ be the smallest integer satisfying $M_{j_s}:=\{\alpha\in L_{j_s}(P)\cap N_{s-1} : \deg(x_\alpha)\neq \eb_{i_{s-1}}\}\neq \emptyset$ if such an integer exists, let $\gamma_s$ be the unique element in $M_{j_s}$, let $N_s:=\{\alpha \in L(P) : \gamma_s\subset \alpha\}$ where we let $N_0:=L(P)$, and let $i_s$ be the integer satisfying $\deg(x_{\gamma_s})=\eb_{i_s}$.
    Since $P$ is finite, this procedure must eventually terminate.
    Let $\ell$ be the largest integer with $M_{j_\ell}\neq \emptyset$.
    By the same argument as above, we can see that $\gamma_s$ is join-irreducible in $L(P)$ and $\deg(x_\alpha)=\eb_{i_{s-1}}$ for any $s=1,\ldots,\ell+1$ and $\alpha\in N_{s-1}\setminus N_s$ where we let $N_{\ell+1}:=\emptyset$.

    Since $\gamma_s$ is join-irreducible and $\gamma_1\subset \cdots \subset \gamma_\ell$, the poset ideal $\gamma_s$ corresponds to an element $c_s$ in $P$ and $C:=\{c_1, \ldots,c_\ell\}$ forms a chain in $P$ with $c_1<_P\cdots<_Pc_\ell$.
    We define the map $f_C:\fkp(C)\to [0,m]$ by $f_C(C_s)=\eb_{i_s}$ for $s=0,1,\ldots,\ell$.
    Then we can see that $\deg_{f_C}(x_\alpha)=\eb_{f_C(C_{s-1})}=\eb_{i_{s-1}}$ for any $s=1,\ldots,\ell+1$ and $\alpha\in N_{s-1}\setminus N_s$.
    Therefore, we have $\deg=\deg_{f_C}$.
\end{proof}

\section{Hilbert series of multigraded Hibi rings}\label{section:Hilbert series of multigraded Hibi rings}
In this section, we compute the multigraded Hilbert series of Hibi rings.

Let $S$ be a standard $\ZZ^{m+1}$-graded polynomial ring.
For a homogeneous ideal $I$, the residue ring
$S/I = \bigoplus_{\ab \in \intnonneg^{m+1}} S_{\ab}/I_{\ab}$
can be considered as a $\ZZ^{m+1}$-graded ring.
The \textit{multigraded Hilbert function} and the \textit{multigraded Hilbert series} of $S/I$ are defined to be
\begin{align*}
  \HF(S/I; \ab) &= \dim_{\kk}(S/I)_{\ab}, \text{ and } \\
  \HS(S/I; \tb) &= \sum_{\ab \in \intnonneg^{m+1}}
  \dim_{\kk}(S/I)_{\ab} \tb^{\ab}, \text{ respectively.}
\end{align*}

To calculate the Hilbert series of Hibi rings, we use the theory of Stanley's $P$-partition (\cite{StanleyP-partitions}, see also \cite[Chapter~4.5]{StanleyCombinatorics}).
Let $P=[1,n]$ be a poset equipped with a partial order $\leq_P$.
We assume that $P$ is compatible with the usual ordering of $[1,n]$, that is, if $i<_Pj$, then $i<j$ in $\ZZ$.
Let $\fkS_n$ be the symmetric group of degree $n$ and let $e(P):=\{(p_1,\ldots,p_n)\in \fkS_n : i<j \text{ if }p_i <_P p_j\}$, which is called the \textit{Jordan-H{\" o}lder set} of $P$.
A map $\sigma : P\to \ZZ_{\ge 0}$ is \textit{order-reversing} if $i\le_P j$ implies $\sigma(i)\ge \sigma(j)$ in $\ZZ_{\ge 0}$.
Let $\calA(P)$ be the set of order-reversing maps.
We can construct an order-reversing map from a poset ideal of $P$; for $\alpha\in L(P)$, we define an order-reversing map $\sigma_{\alpha} : P\to \ZZ_{\ge 0}$ by setting $\sigma_\alpha(i)=1$ if $i\in \alpha$ and $\sigma_\alpha(i)=0$ otherwise.
Then there is the following bijection for any $z\in \ZZ_{\ge 0}$ (cf. \cite[Proposition~2.3]{Hibi1991Hilbert}):
\begin{equation}\label{equ:1to1order}
    \begin{split}
    \{x_{\alpha_1}\cdots x_{\alpha_z}\in \kk[P] :\alpha_1,\ldots,\alpha_z\in L(P)\} &\to \{\sigma\in \calA(P) : \sigma(i)\leq z \text{ for any }i\in P\}, \\  
    x_{\alpha_1}\cdots x_{\alpha_z} &\mapsto \sigma=\sigma_{\alpha_1}+\cdots+\sigma_{\alpha_z}.
    \end{split}
\end{equation}

Let $C:=\{c_1<_P\cdots<_Pc_\ell\}$ be a chain in $P$.
We first consider the Hilbert series of $\kk[P]$ multigraded by $C$.
If a monomial $x_{\alpha_1}\cdots x_{\alpha_z}$ ($\alpha_1,\ldots,\alpha_z\in L(P)$) has the degree $\ab=(a_0,a_1,\ldots,a_\ell)$, then we have $|\{\alpha_j : \alpha_j\cap C=C_i\}|=a_i$ for each $i=0,1,\ldots,\ell$,
equivalently, if we set $\sigma=\sigma_{\alpha_1}+\cdots+\sigma_{\alpha_z}$, then we obtain
\begin{equation*}
    \begin{split}
        a_\ell&=\sigma(c_\ell), \\
        a_{\ell-1}&=\sigma(c_{\ell-1})-\sigma(c_\ell), \\
        & \; \; \vdots \\
        a_1&=\sigma(c_1)-\sigma(c_2), \\
        a_0&= z-\sigma(c_1).
    \end{split}
\end{equation*}
Therefore, it follows from (\ref{equ:1to1order}) that for any $\ab\in \ZZ^{\ell+1}$, we have the following one-to-one correspondence:

\begin{equation}\label{equ:a1to1}
    \begin{split}
    \{x_{\alpha_1}\cdots x_{\alpha_{s(\ab)}}\in \kk[P] :\;  \alpha_1,
    \ldots,\alpha_{s(\ab)}&\in L(P), \;\deg_C(x_{\alpha_1}\cdots x_{\alpha_{s(\ab)}})=\ab\} \\
    &\updownarrow \\
    \{\sigma\in \calA(P) : \sigma(i)\leq s(\ab) \text{ for any }i\in P,\; &\sigma(c_j)=a_\ell+\cdots+a_{j} \text{ for any }j\in[1,\ell]\},
    \end{split}
\end{equation}
where $s(\ab):=a_0+\cdots+a_\ell$.

For $\pi=(p_1,\ldots,p_n)\in \fkS_n$, we say that a function $f:[1,n]\to \ZZ_{\ge 0}$ is \textit{$\pi$-compatible} if $f$ satisfies the following conditions:
\begin{itemize}
    \item[(i)] $f(p_1)\ge f(p_2)\ge \cdots \ge f(p_n)$ and
    \item[(ii)] $f(p_i)>f(p_{i+1})$ if $p_i>p_{i+1}$. 
\end{itemize}
For any function $f:[1,n]\to \ZZ_{\ge 0}$, there is a unique permutation $\pi\in \fkS_n$ such that $f$ is $\pi$-compatible (\cite[Lemma~4.5.1]{StanleyCombinatorics}).
Let $S_\pi$ be the set of all $\pi$-compatible functions.

\begin{lemma}[{\cite[Lemma~4.5.3]{StanleyCombinatorics}}]\label{lemma:p-partition}
One has $\calA(P)=\bigsqcup_{\pi\in e(P)}S_\pi$.
\end{lemma}

For $\pi=(p_1,\ldots,p_n)\in \fkS_n$ and $i\in [1,n]$, let $d_i(\pi):=|\{j\in[i,n] 
 : p_j>p_{j+1}\}|$ where we let $p_{n+1}:=n+1$.
In addition, let $\ind(\pi,c_j)$ be the integer satisfying $c_j=p_{\ind(\pi,c_j)}$.
Note that $\ind(\pi,c_1)<\ind(\pi,c_2)<\cdots<\ind(\pi,c_\ell)$ if $\pi\in e(P)$.

\medskip

We are now ready to compute the Hilbert series of $\kk[P]$.

\begin{theorem}\label{theorem:Hilbert series of multigraded Hibi rings}
Let $P=[1,n]$ be a poset and let $C=\{c_1<_P\cdots<_Pc_\ell\}$ be a chain in $P$.
Suppose that $\kk[P]$ is multigraded by $C$.
Then we have
\begin{align}\label{ali:HSHibi}
    \HS(\kk[P]; \tb)=\sum_{\pi\in e(P)}\prod_{i\in[0,\ell]}\frac{t_i^{d_{\ind(\pi,c_{i})}(\pi)-d_{\ind(\pi,c_{i+1})}(\pi)}}{(1-t_i)^{\ind(\pi,c_{i+1})-\ind(\pi,c_i)}},
\end{align}
where we let $\ind(\pi,c_0):=0$, $\ind(\pi,c_{\ell+1}):=n+1$, $d_0(\pi):=d_1(\pi)$ and $d_{n+1}(\pi):=0$.
\end{theorem}
\begin{proof}
    For $\ab\in \ZZ^{\ell+1}_{\ge 0}$, let $\calA(P)_\ab:=\{\sigma\in \calA(P) : \sigma(i)\leq s(\ab) \text{ for any }i\in P,\; \sigma(c_j)=a_\ell+\cdots+a_{j} \text{ for any }j\in[1,\ell]\}$.
    From the one-to-one correspondence (\ref{equ:a1to1}), we obtain $\HF(\kk[P];\ab)=|\calA(P)_\ab|$.
    Moreover, for $\pi\in e(P)$ and a function $\sigma: [1,n]\to \ZZ_{\ge 0}$, we can see that $\sigma\in S_\pi\cap \calA(P)_\ab$ if and only if $\sigma$ satisfies 
    \begin{equation}\label{equ:ineq}
        \begin{split}
        &\sigma(c_j)=a_\ell+a_{\ell-1}+\cdots+a_j \text{ for any }j\in [1,\ell] \text{ and}\\
        &s(\ab)-d_0(\pi)\ge \sigma(p_1)-d_1(\pi)\ge \sigma(p_2)-d_2(\pi)\ge \cdots \ge \sigma(p_n)-d_n(\pi)\ge 0.
        \end{split}
    \end{equation}
     Therefore, the number of functions satisfying condition (\ref{equ:ineq}) is
     \begin{multline*}
     \prod_{i\in[0,\ell]}\left( \!\! \binom{(\sigma(c_i)-d_{\ind(\pi,c_i)}(\pi))-(\sigma(c_{i+1})-d_{\ind(\pi,c_{i+1})}(\pi))+1}{\ind(\pi,c_{i+1})-\ind(\pi,c_i)-1} \!\! \right)\\=
     \prod_{i\in[0,\ell]}\left( \!\! \binom{a_i-d_{\ind(\pi,c_i)}(\pi)+d_{\ind(\pi,c_{i+1})}(\pi)+1}{\ind(\pi,c_{i+1})-\ind(\pi,c_i)-1} \!\! \right),
     \end{multline*}
     where we let $\sigma(c_0):=s(\ab)$.
     By Lemma~\ref{lemma:p-partition}, we get
     \begin{align*}
         \HS(\kk[P];\tb)&=\sum_{\ab\in \ZZ^{\ell+1}_{\geq 0}}\left(\sum_{\pi\in e(P)}|S_\pi\cap \calA(P)_\ab|\right)\tb^\ab \\
         &=\sum_{\pi\in e(P)}\prod_{i\in[0,\ell]}\left(\sum_{a_i\in \ZZ_{\geq 0}} \left( \!\! \binom{a_i-d_{\ind(\pi,c_i)}(\pi)+d_{\ind(\pi,c_{i+1})}(\pi)+1}{\ind(\pi,c_{i+1})-\ind(\pi,c_i)-1} \!\! \right)t_i^{a_i}\right) \\
         &=\sum_{\pi\in e(P)}\prod_{i\in[0,\ell]}\frac{t_i^{d_{\ind(\pi,c_{i})}(\pi)-d_{\ind(\pi,c_{i+1})}(\pi)}}{(1-t_i)^{\ind(\pi,c_{i+1})-\ind(\pi,c_i)}}.
     \end{align*}
\end{proof}

Let $m$ be a nonnegative integer and let $f:\fkp(C)\to [0,m]$ be a map.
In the case where $S_{L(P)}$ is multigraded by $f_C$, we can calculate the Hilbert series of $\kk[P]$ by replacing $t_i$ with $t_{f(C_i)}$ in (\ref{ali:HSHibi}), so we immediately get the following corollary:
\begin{corollary}\label{corollary:Hilbert series}
     Work with the same notation as above. Suppose that $S_{L(P)}$ is multigraded by $f_C$.
     Then we have 
     \[
     \HS(\kk[P]; \tb)=\sum_{\pi\in e(P)}\prod_{i\in[0,\ell]}\frac{t_{f(C_i)}^{d_{\ind(\pi,c_{i})}(\pi)-d_{\ind(\pi,c_{i+1})}(\pi)}}{(1-t_{f(C_i)})^{\ind(\pi,c_{i+1})-\ind(\pi,c_i)}}.
     \]
\end{corollary}
In particular, when $m=0$, we obtain the Hilbert series of a $\ZZ$-graded Hibi ring, which coincides with the result of \cite[Proposition~2.3]{Hibi1991Hilbert}.

\begin{example}\label{example:multigraded Hilbert series of N-poset}
  Let $P$ be given by Figure~\ref{fig:N-poset}.
  Then we have 
  \[e(P)=\{(1,2,3,4),(1,2,4,3),(2,1,3,4),(2,1,4,3),(2,4,1,3)\}.\]
  We compute the Hilbert series of $\kk[P]$ using Theorem~\ref{theorem:Hilbert series of multigraded Hibi rings}.
  If $S_{L(P)}$ is multigraded by $C_1:=\{2,3\}$, then 
  \begin{align*}
   \HS(\kk[P]; \tb)=&\frac{1}{(1-t_0)^2(1-t_1)(1-t_2)^2}+\frac{t_1}{(1-t_0)^2(1-t_1)^2(1-t_2)}+ \\
   &\frac{t_1}{(1-t_0)(1-t_1)^2(1-t_2)^2}+\frac{t_1^2}{(1-t_0)(1-t_1)^3(1-t_2)}+\frac{t_1}{(1-t_0)(1-t_1)^3(1-t_2)} \\
   =& \frac{1 + t_1 - 2 t_0 t_1 - 2 t_1 t_2 + t_0 t_1 t_2 + t_0 t_1^2 t_2}{(1-t_0)^2(1-t_1)^3(1-t_2)^2} \\
   =& \frac{1 - 2 t_0 t_1 - t_1^2 - 2 t_1 t_2 + 2 t_0 t_1^2 + t_0 t_1 t_2 + 2 t_1^2 t_2 - t_0 t_1^3 t_2}{(1-t_0)^2(1-t_1)^4(1-t_2)^2}.
  \end{align*}
  On the other hand, if $S_{L(P)}$ is multigraded by $C_2:=\{2,4\}$, then
  \begin{align*}
   \HS(\kk[P]; \tb)=&\frac{1}{(1-t_0)^2(1-t_1)^2(1-t_2)}+\frac{t_2}{(1-t_0)^2(1-t_1)(1-t_2)^2}+ \\
   &\frac{t_1}{(1-t_0)(1-t_1)^3(1-t_2)}+\frac{t_1t_2}{(1-t_0)(1-t_1)^2(1-t_2)^2}+\frac{t_2}{(1-t_0)(1-t_1)(1-t_2)^3} \\
   =& \frac{1 - t_0 t_1 - t_0 t_2 - 3 t_1 t_2 + 3 t_0 t_1 t_2 + t_1^2 t_2 + t_1 t_2^2 - t_0 t_1^2 t_2^2}
   {(1-t_0)^2(1-t_1)^3(1-t_2)^3}.
  \end{align*}
  
  In either case, by setting $t_0=t_1=t_2=t$, we obtain the Hilbert series of the standard $\ZZ$-graded Hibi ring of $P$ as follows:
  \begin{align*}
   \HS(\kk[P]; t,t,t)= \frac{1 - 5t^2 + 5t^3 - t^5}{(1-t)^8}=\frac{1+3t+t^2}{(1-t)^5}.
  \end{align*}
\end{example}

\begin{example}\label{example:multigraded Hilbert series of Boolean lattice}
    Let $P:=[1,n]$ be an $n$-element antichain, that is, $i$ and $j$ are incomparable for any two distinct elements $i,j\in P$.
    Let $C:=\{n\}$ and suppose that $S_{L(P)}$ is multigraded by $C$.
    From Theorem~\ref{theorem:Hilbert series of multigraded Hibi rings}, we have
    \[
    \HS(\kk[P]; \tb)=\sum_{\pi\in \fkS_n}\frac{t_0^{d_0(\pi)-d_{\ind(\pi,n)}(\pi)}t_1^{d_{\ind(\pi,n)}(\pi)}}{(1-t_0)^{\ind(\pi,n)}(1-t_1)^{n+1-\ind(\pi,n)}}.
    \]

    For $k\in [1,n]$, $k_1\in[0,k-2]$ and $k_2\in[0,n-k-1]$, we consider a permutation $\pi=(p_1,\ldots,p_n)\in\fkS_n$ such that $p_k=n$, $|\{i\in [1,k-1] : p_i>p_{i+1}\}|=k_1$ and $|\{i\in [k+1,n] : p_i>p_{i+1}\}|=k_2$.
    The number of such permutations is $\binom{n-1}{k-1}A_{k_1,k-1}A_{k_2,n-k}$ where $A_{s,r}$ denotes the \textit{Eulerian number}, which is the number of permutations of the numbers $1$ to $r$ in which exactly $s$ elements are less than the previous element. 
    Therefore, we obtain
    \begin{align*}
    \HS(\kk[P]; \tb)&=\frac{\sum_{k_1=0}^{n-2}A_{k_1,n-1}t_0^{k_1}}{(1-t_0)^n(1-t_1)}+\sum_{k=1}^{n-1}\frac{\binom{n-1}{k-1}\sum_{k_1=0}^{k-2}\sum_{k_2=0}^{n-k-1}A_{k_1,k-1}A_{k_2,n-k}t_0^{k_1}t_1^{k_2+1}}{(1-t_0)^k(1-t_1)^{n+1-k}} \\
    &=\frac{A_{n-1}(t_0)}{(1-t_0)^n(1-t_1)}+\sum_{k=1}^{n-1}\frac{\binom{n-1}{k-1}A_{k-1}(t_0)A_{n-k}(t_1)t_1}{(1-t_0)^k(1-t_1)^{n+1-k}},
    \end{align*}
    where $A_r(t):=\sum_{s\in[0,r-1]}A_{s,r}t^s$ (we define $A_0(t):=1$) denotes the \textit{Eulerian polynomial}.

    By setting $t_0=t_1=t$, we can compute the Hilbert series of the standard $\ZZ$-graded Hibi ring of $P$ as follows:
    \[
    \HS(\kk[P]; t,t)=\frac{A_{n-1}(t)+\sum_{k=1}^{n-1}\binom{n-1}{k-1}A_{k-1}(t)A_{n-k}(t)t}{(1-t)^{n+1}}=\frac{A_n(t)}{(1-t)^{n+1}},
    \]
    where we used the recurrence relation $A_n(t)=A_{n-1}(t)+\sum_{k=1}^{n-1}\binom{n-1}{k-1}A_{k-1}(t)A_{n-k}(t)t$ for the Eulerian polynomials (see, e.g., \cite[Theorem~1.5]{Petersen2015Eulerian}).
\end{example}

\section{Multidegree polynomials of multigraded Hibi rings}\label{section:Multidegree polynomials of multigraded Hibi rings}
In this section, we compute multidegree polynomials of multigraded Hibi rings.
For a standard $\ZZ^{\ell+1}$-graded polynomial ring $S = \kk[x_1, \ldots, x_m]$ and a homogeneous ideal $I$, the Hilbert series of $S/I$ can be
written as
\begin{equation*}
    \HS(S/I; \tb) = \frac{\calK(S/I; \tb)}{\prod_{i=1}^{m}(1-\tb^{\deg(x_i)})}.
\end{equation*}
The numerator $\calK(S/I; \tb) \in \ZZ[\tb]$ is called the \textit{$\calK$-polynomial} of $S/I$.

\begin{definition}[See, e.g.,~{\cite[Definition 8.45]{MillerSturmfels}}]\label{definition:multidegree}
    The \textit{multidegree polynomial} of a $\ZZ^{\ell+1}$-graded ring $S/I$ is the sum $\calC(S/I; \tb) \in \ZZ[\tb]$
    of all terms in
    \begin{equation*}
        \calK(S/I; 1-\tb) \coloneqq \calK(S/I; 1-t_0, 1-t_1, \ldots, 1-t_{\ell})
    \end{equation*}
    having total degree $\codim(S/I)$.
\end{definition}

Multidegree polynomials have two characterizing properties.
Firstly, multidegree polynomials are additive.
For a prime ideal $\fkp$, let $\length_{S_{\fkp}}(S/I)$ denote the length of $(S/I)_{\fkp}$ as an $S_{\fkp}$-module.
For a standard $\ZZ^{\ell+1}$-graded polynomial ring $S = \kk[x_1, \ldots, x_m]$ and a homogeneous ideal $I$,
we have
\begin{equation*}
    \calC(S/I; \tb) = \sum_{\fkp} \length_{S_{\fkp}}(S/I) \cdot \calC(S/\fkp; \tb),
\end{equation*}
where $\fkp$ are associated primes of $S/I$ with $\dim(S/\fkp) = \dim(S/I)$.
Secondly, multidegree polynomials are degenerative. For a monomial order $\preceq$ on $S$,
we have
\begin{equation*}
    \calC(S/I; \tb) = \calC(S/\ini_{\preceq}(I); \tb).
\end{equation*}

The multidegree polynomial of a prime ideal generated by variables can be easily computed.
\begin{theorem}[See, e.g.,~{\cite[Theorem 8.44]{MillerSturmfels}}]\label{theorem:multidegree of prime ideal}
    Let $x_{i_1}, \ldots, x_{i_r}$ are variables in $S$. Then
    \begin{equation*}
        \calC(S/(x_{i_1}, \ldots, x_{i_r}); \tb) = \langle \deg(x_{i_1}), \tb \rangle \cdots \langle \deg(x_{i_r}), \tb \rangle,
    \end{equation*}
    where $\langle \cdot, \cdot \rangle$ denotes the usual inner product.
\end{theorem}

For more details on multidegree polynomials, see~\cite[Section 8.5]{MillerSturmfels}.

Now we compute multidegree polynomials of multigraded Hibi rings.
We use the same terminology as in Section~\ref{section:Hilbert series of multigraded Hibi rings}.
In addition, let $\scrM_{L(P)}$ be the set of maximal chains of $L(P)$ and let $\calM^c:=L(P)\setminus \calM$ for $\calM\in \scrM_{L(P)}$.

\begin{proposition}\label{proposition:multidegree of Hibi rings}
    Let $P=[1,n]$ be a poset and let $C=\{c_1<_P\cdots<_Pc_\ell\}$ be a chain in $P$.
    Suppose that $\kk[P]$ is multigraded by $C$.
    Then, the multidegree polynomial of $\kk[P]$ is
    \begin{equation*}
        \calC(\kk[P]; \tb) = \sum_{\calM\in \scrM_{L(P)}}\tb^{\deg_{C}({\calM^c})},
    \end{equation*}
    where $\deg_{C}({\calM^c}) = \deg_C(\prod_{\alpha \in \calM^c}x_{\alpha})$.
\end{proposition}
\begin{proof}
    Since multidegree polynomials are degenerative, we compute $\calC(S_{L(P)}/\ini_{\preceq}(I_{L(P)}); \tb)$ with a compatible monomial
    order $\preceq$. We regard $\ini_{\preceq}(I_{L(P)})$ as a Stanley-Reisner ideal associated with a simplicial
    complex $\Delta_P$ with the vertex set $L(P)$. Since $\ini_{\preceq}(I_{L(P)})$ is generated by $x_{\alpha}x_{\beta}$
    where $\alpha, \beta \in L(P)$ are incomparable, the simplicial complex $\Delta_P$ consists of chains of $L(P)$.
    In particular, the facets of $\Delta_P$ correspond to maximal chains of $L(P)$. Therefore, we have the primary decomposition
    \begin{equation*}
        \ini_{\preceq}(I_{L(P)}) = \bigcap_{\calM\in \scrM_{L(P)}} \fkp_{\calM^c},
    \end{equation*}
    where $\fkp_{\calM^c} = (x_{\alpha}: \alpha \in \calM^c)$.
    Since the distributive lattice $L(P)$ is pure (i.e., all its maximal chains have the same length), the initial ideal
    $\ini_{\preceq}(I_{L(P)})$ is unmixed. We can easily check that, for each maximal chain $\calM$, the length of 
    $(S_{L(P)}/\ini_{\preceq}(I_{L(P)}))_{\fkp_{\calM^c}}$ is equal to $1$.
    Thus, by the additivity of multidegree polynomials and by Theorem~\ref{theorem:multidegree of prime ideal}, we have
    \begin{equation*}
        \calC(S_{L(P)}/\ini_{\preceq}(I_{L(P)}); \tb) = \sum_{\calM\in \scrM_{L(P)}}\calC(S_{L(P)}/\fkp_{\calM^c}; \tb) = \sum_{\calM\in \scrM_{L(P)}}\tb^{\deg_{C}({\calM^c})}.
    \end{equation*}
\end{proof}

\begin{corollary}\label{corollary:multidegree polynomial with f}
     Work with the same notation as above. Suppose that $S_{L(P)}$ is multigraded by $f_C$.
     Then we have 
     \begin{equation*}
         \calC(\kk[P]; \tb) = \sum_{\calM\in \scrM_{L(P)}}\tb^{\deg_{f_C}({\calM^c})},
     \end{equation*}
     where $\deg_{f_C}({\calM^c}) = \deg_{f_C}(\prod_{\alpha \in \calM^c}x_{\alpha})$.
\end{corollary}
In particular, when $m=0$, the multidegree polynomial $\calC(\kk[P]; t_0)$ consists of exactly one term. Its coefficient is the number of maximal chains of $L(P)$, which coincides with the degree
of $\kk[P]$.

\begin{example}\label{multidegree of N-poset}
    Let $P$ be given by Figure~\ref{fig:N-poset}.
  Then $L(P)$ has the following 5 maximal chains:
  \begin{align*}
  \emptyset\subset \{1\}\subset \{1,2\}\subset \{1,2,3\} \subset\{1,2,3,4\}, \\
  \emptyset\subset \{1\}\subset \{1,2\}\subset \{1,2,4\} \subset\{1,2,4,3\}, \\
  \emptyset\subset \{2\}\subset \{2,1\}\subset \{2,1,3\} \subset\{2,1,3,4\}, \\
  \emptyset\subset \{2\}\subset \{2,1\}\subset \{2,1,4\} \subset\{2,1,4,3\}, \\
  \emptyset\subset \{2\}\subset \{2,4\}\subset \{2,4,1\} \subset\{2,4,1,3\}.
  \end{align*}
  From Proposition~\ref{proposition:multidegree of Hibi rings}, if $S_{L(P)}$ is multigraded by $C_1:=\{2,3\}$, then we have
  \[
  \calC(\kk[P];\tb)=t_1^3+t_1^2t_2+t_0t_1^2+2t_0t_1t_2.
  \]
  On the other hand, if $S_{L(P)}$ is multigraded by $C_2:=\{2,4\}$, the we have
  \[
  \calC(\kk[P];\tb)=t_1t_2^2+t_1^2t_2+t_0t_2^2+t_0t_1t_2+t_0t_1^2.
  \]
    In either case, by setting $t_0=t_1=t_2=t$, we obtain
  \begin{align*}
   \calC(\kk[P];t,t,t) = 5t^3.
  \end{align*}
  The degree of $5t^3$ coincides with the codimension of $\kk[P]$ and
  its coefficient $5$ is the degree of the $\ZZ$-graded Hibi ring $\kk[P]$.
\end{example}

\begin{example}\label{multidegree of Boolean lattice}
    We revisit the poset $P$ and multigradings described in Example~\ref{example:multigraded Hilbert series of Boolean lattice}.
    Note that for $\alpha\in L(P)$, we have $\deg_C(x_\alpha)=\eb_0$ if $n\notin \alpha$ and $\deg_C(x_\alpha)=\eb_1$ otherwise.
    A permutation $\pi=(p_1,\ldots,p_n)\in \fkS_n$ corresponds to a maximal chain $\calM_\pi:=\{\emptyset\subset \{p_1\}\subset \{p_1,p_2\}\subset \cdots \subset\{p_1,\ldots,p_n\}\}$ of $L(P)$ (this correspondence is bijection).
    For any $k\in [1,n]$, the cardinality of $\fkS_{n,k}:=\{(p_1,\ldots,p_n)\in \fkS_n : p_k=n\}$ is equal to $(n-1)!$ and $\tb^{\deg_C(\calM_\pi)}=t_0^{k}t_1^{n+1-k}$ for any $\pi\in \fkS_{n,k}$.
    Therefore, by Proposition~\ref{proposition:multidegree of Hibi rings}, we have
    \begin{equation*}
        \calC(\kk[P]; \tb) = \sum_{\pi\in \fkS_n}\tb^{\deg_{C}(\calM_\pi^c)}=\sum_{k=1}^n\sum_{\pi\in \fkS_{n,k}}\tb^{\deg_{C}(\calM_\pi^c)}=(n-1)!\sum_{k=1}^nt_0^{2^{n-1}-k}t_1^{2^{n-1}-n-1+k}.
    \end{equation*}

    By setting $t_0=t_1=t$, we obtain
    \begin{equation*}
        \calC(\kk[P]; t,t) = (n-1)!\sum_{k=1}^nt^{2^n-n-1} = n!t^{2^n-n-1}.
    \end{equation*}
    The degree of $n!t^{2^n-n-1}$ coincides with the codimension of $\kk[P]$ and
    its coefficient $n!$ is the degree of the $\ZZ$-graded Hibi ring $\kk[P]$.
\end{example}

\section{When do Hibi ideals become Cartwright-Sturmfels?}\label{section:When do Hibi ideals become Cartwright-Sturmfels}
First, we review definitions and known results about Cartwright-Sturmfels ideals.
Let $n$ and $m_1, \ldots, m_n$ be positive integers.
Let $S = \kk[x_{i,j} : 1 \leq j \leq n, 1 \leq i \leq m_j]$ be a standard $\ZZ^n$-graded
polynomial ring with $\deg(x_{i,j}) = \eb_j$.

Let $G = \GL_{m_1}(\kk) \times \cdots \times \GL_{m_n}(\kk)$.
We assume that $G$ acts on $S$.
Precisely, let $g = (g^{(1)}, \ldots, g^{(n)}) \in G$ acts on variables in $S$ by
\begin{equation*}
  g \cdot x_{i,j} = \sum_{k=1}^{m_j}g_{k,i}^{(j)} x_{k,j}.
\end{equation*}
Let $B = \borel_{m_1}(\kk) \times \cdots \times \borel_{m_n}(\kk)$ be the \textit{Borel subgroup}
of $G$, where each $\borel_{m_j}(\kk)$ consists of upper triangular matrices in $\GL_{m_j}(\kk)$.
We say that a homogeneous ideal $J \subset S$ is \textit{Borel-fixed} if $g \cdot J = J$
for all $g \in B$.

\begin{definition}[{\cite[Definition 2.4]{ConcaDeNegriGorlaUGBandCS}}]\label{definition:Cartwright-Sturmfels ideal}
  A homogeneous ideal  $I$ is \textit{Cartwright-Sturmfels} if there exists a squarefree
  multigraded Borel-fixed monomial ideal $J$ such that $\HS(S/I; \tb) = \HS(S/J; \tb)$.
\end{definition}

By Definition~\ref{definition:Cartwright-Sturmfels ideal}, a homogeneous ideal $I$ is
a Cartwright-Sturmfels ideal if and only if $\ini_{\preceq}(I)$ is a Cartwright-Sturmfels
ideal for a monomial order because $\HS(S/I; \tb) = \HS(S/\ini_{\preceq}(I); \tb)$.
It was shown that, if $\kk$ is an infinite field, Cartwright-Sturmfels ideals are characterized
with multigraded generic initial ideals.
As in the $\ZZ$-graded case, for a homogeneous ideal $I$ and a monomial order $\preceq$,
there exists a Zariski open dense set $U \subset G$ such that the initial ideal $\ini_{\preceq}(g \cdot I)$
stays the same for all $g \in U$. We denote such an initial ideal by $\gin_{\preceq}(I)$
and call it the \textit{multigraded generic initial ideal} of $I$ with respect to $\preceq$.

\begin{proposition}[{\cite[Proposition 2.6]{ConcaDeNegriGorlaUGBandCS}}]\label{proposition:generic initial ideal of Cartwright-Sturmfels ideal}
  Let $I \subset S$ be a Cartwright-Sturmfels ideal and let $J \subset S$ be a squarefree
  Borel-fixed monomial ideal such that $\HS(S/I; \tb) = \HS(S/J; \tb)$.
  Then, $\gin_{\preceq}(I) = J$ for all monomial order $\preceq$.
  In particular, $I$ is radical and is generated by polynomials whose degree
  in each component is either $0$ or $1$.
\end{proposition}

As a corollary of Proposition~\ref{proposition:generic initial ideal of Cartwright-Sturmfels ideal},
if $I$ is a Cartwright-Sturmfels ideal, then $\ini_{\preceq}(I)$ is a squarefree monomial ideal generated by monomials
whose degree in each component is either $0$ or $1$.

Cartwright-Sturmfels ideals are closed under elimination of variables.

\begin{theorem}[{\cite[Theorem 2.16 (6)]{ConcaDeNegriGorlaUGBandCS}}]\label{theorem:eliminated ideal of Cartwright-Sturmfels ideal}
  Let $V_i \subset S_{\eb_i}$ be vector subspace for $i = 1, \ldots, n$
  and let $R = \kk[V_1, \ldots, V_n]$ be the $\ZZ^n$-graded subring of $S$.
  If $I$ is a Cartwright-Sturmfels ideal of $S$, then $I \cap R$ is also a
  Cartwright-Sturmfels ideal of $R$.
\end{theorem}

Theorem~\ref{theorem:2-minors is Cartwright-Sturmfels ideal} was proved by
Cartwright-Sturmfels in~\cite{CSHilbertScheme}. Inspired by their work,
Conca-De Negri-Gorla introduced the term ``Cartwright-Sturmfels ideal'' and studied it.
Theorem~\ref{theorem:2-minors is Cartwright-Sturmfels ideal} plays an important role in the proof of
our main theorem about the relationship between Hibi ideals and Cartwright-Sturmfels ideals.

\begin{theorem}[{\cite{CSHilbertScheme}}]\label{theorem:2-minors is Cartwright-Sturmfels ideal}
  Let $S = \kk[x_{i,j} : 1 \leq j \leq n, 1 \leq i \leq m]$ be a standard $\ZZ^n$-graded
  polynomial ring with $\deg(x_{i,j}) = \eb_j$. Let $X = (x_{i, j})$ be a matrix of variables
  and let $I_2(X)$ be a homogeneous ideal of $S$ generated by all $2$-minors of $X$.
  Then, $I_2(X)$ is a Cartwright-Sturmfels ideal.
\end{theorem}

Now, we characterize Hibi ideals that are Cartwright-Sturmfels ideals.
For simplicity, we consider the case that $S_{L(P)}$ is multigraded by a chain of $P$.

\begin{theorem}\label{theorem:characterization of CS Hibi ideal}
  Let $P$ be a poset and $L(P)$ be the distributive lattice associated with $P$.
  Let $C = \{c_1 <_P \ldots <_P c_{\ell}\}$ be a chain of $P$.
  We assume that $S_{L(P)}$ is multigraded by $C$.
  Then, $I_{L(P)}$ is a Cartwright-Sturmfels ideal if and only if $P \setminus C$
  is a chain of $P$.
\end{theorem}
\begin{proof}
  First, we assume $P \setminus C$ is not a chain of $P$.
  Let $a$ and $b$ are incomparable elements of $P \setminus C$.
  Let $\alpha$ and $\beta$ are poset ideals of $P$ defined to be
  \begin{equation*}
    \alpha = \{p \in P : p \leq_P a\}, \text{ and } \beta = \{p \in P : p \leq_P b\},
  \end{equation*}
  respectively.
  We assume $\alpha \cap C = C_i$ and $\beta \cap C = C_j$.
  Furthermore, we can assume $i \leq j$ without loss of generality.
  Let $\alpha' = \alpha \cup C_j$.
  Then, $\alpha'$ and $\beta$ are poset ideals, in other words, one has $\alpha', \beta \in L(P)$.
  Since $a$ and $b$ are incomparable and $c_j \leq b$, one has $a \in \alpha'$ and
  $b \notin \alpha'$. Similarly, since $a$ and $b$ are incomparable,
  one has $b \in \beta$ and $a \notin \beta$. Therefore, neither $\alpha'$ nor
  $\beta$ includes the other, i.e., $\alpha'$ and $\beta$ are incomparable in $L(P)$.
  Since $\calF_{L(P)}$ is a \groebner basis of $I_{L(P)}$ with respect to a compatible
  monomial order $\preceq$, the monomial $x_{\alpha'}x_{\beta}$ is a monomial in
  the minimal generators of $\ini_{\preceq}(I_{L(P)})$.
  However, by the definition of $\alpha'$ and $\beta$, one has
  $\deg(x_{\alpha'}x_{\beta}) = 2 \eb_j$.
  By the corollary of Proposition~\ref{proposition:generic initial ideal of Cartwright-Sturmfels ideal},
  the ideal $I_{L(P)}$ is not a Cartwright-Sturmfels ideal.

  Next, we assume $P \setminus C$ is a chain of $P$.
  Let $D = P \setminus C = \{d_1 <_P \ldots <_P d_r\}$.
  We consider the case where $d_i$ and $c_j$ are incomparable for all $i = 1, \ldots, r$ and $j = 1, \ldots, l$.
  Then, the poset $P$ and the distributive lattice $L(P)$ are depicted as in
  Figure~\ref{fig:poset of 2-minors} and Figure~\ref{fig:distributive lattice of 2-minors}, respectively.
  The poset ideals of $P$ can be indexed by $i = 0, 1, \ldots, r$ and $j = 0, 1, \ldots, l$, where, we set
  \begin{equation*}
      \alpha_{i,j} = D_i \cup C_j.
  \end{equation*}
  Under this setting, $\deg(x_{\alpha_{i,j}}) = \eb_j$.
  If $\alpha_{i_1, j_1}$ and $\alpha_{i_2, j_2}$ are incomparable,
  one has $i_1 < i_2, j_2 < j_1$ or $i_2 < i_1, j_1 < j_2$.
  We assume $i_1 < i_2, j_2 < j_1$ without loss of generality.
  Then, one has
  \begin{equation*}
    \alpha_{i_1, j_1} \cap \alpha_{i_2, j_2} = \alpha_{i_1, j_2} \text{ and }
    \alpha_{i_1, j_1} \cup \alpha_{i_2, j_2} = \alpha_{i_2, j_1}.
  \end{equation*}
  Therefore,
  \begin{equation*}
    \calF_{L(P)} = \{
      f_{\alpha_{i_1, j_1}, \alpha_{i_2, j_2}}
       = x_{\alpha_{i_1, j_1}}x_{\alpha_{i_2, j_2}} - x_{\alpha_{i_1, j_2}}x_{\alpha_{i_2, j_1}}
       : 0 \leq i_1 < i_2 \leq l, 0 \leq j_2 < j_1 \leq r
       \}
  \end{equation*}
  and $I_{L(P)}$ coincides with the ideal generated by the $2$-minors of the column-graded
  $(r+1)\times(l+1)$-matrix $X = (x_{\alpha_{i,j}})$.
  By Theorem~\ref{theorem:2-minors is Cartwright-Sturmfels ideal},
  the ideal $I_{L(P)}$ is a Cartwright-Sturmfels ideal.
  \begin{figure}[ht]
    \begin{minipage}{0.48\columnwidth}
      \centering
      {\scalebox{0.45}{
        \begin{tikzpicture}[line width=0.05cm]
          \coordinate (c1) at (2,-1); 
          \coordinate (c2) at (2,1);
          \coordinate (c3) at (2,7);
          \coordinate (c4) at (2,9);
          \coordinate (d1) at (0,0);
          \coordinate (d2) at (0,2);
          \coordinate (d3) at (0,6);
          \coordinate (d4) at (0,8);

          \draw  (c1)--(c2);
          \draw  (c2)--(2,2);
          \draw[dotted] (2,2.5)--(2,5.5);
          \draw  (2,6)--(c3);
          \draw  (c3)--(c4);
          \draw  (d1)--(d2);
          \draw  (d2)--(0,3);
          \draw[dotted] (0,3.5)--(0,4.5);
          \draw  (0,5)--(d3);
          \draw  (d3)--(d4);
          
          \draw [line width=0.05cm, fill=white] (c1) circle [radius=0.15];
          \draw [line width=0.05cm, fill=white] (c2) circle [radius=0.15];
          \draw [line width=0.05cm, fill=white] (c3) circle [radius=0.15];
          \draw [line width=0.05cm, fill=white] (c4) circle [radius=0.15];
          \draw [line width=0.05cm, fill=white] (d1) circle [radius=0.15];
          \draw [line width=0.05cm, fill=white] (d2) circle [radius=0.15];
          \draw [line width=0.05cm, fill=white] (d3) circle [radius=0.15];
          \draw [line width=0.05cm, fill=white] (d4) circle [radius=0.15];

          \draw [line width=0.02cm, decorate,decoration={brace,amplitude=15pt,mirror}](2.25,-1) -- (2.25,9) node[midway,xshift=0.8cm] {\huge $l$}; 
          \draw [line width=0.02cm, decorate,decoration={brace,amplitude=15pt}](-0.25,0) -- (-0.25, 8) node[black,midway,xshift=-0.8cm] {\huge $r$};
          \node at (0,-1) [font=\huge] {$D$};
          \node at (2,-2) [font=\huge] {$C$};
        \end{tikzpicture}
        }}
        \caption{$P$}
        \label{fig:poset of 2-minors}
      \end{minipage}
      \begin{minipage}{0.48\columnwidth}
        \centering
        {\scalebox{0.45}{
          \begin{tikzpicture}[line width=0.05cm]
            \coordinate (p00) at (0,0);
            \coordinate (p01) at (-1,1);
            \coordinate (p02) at (-3,3);
            \coordinate (p03) at (-4,4);
            \coordinate (p10) at (1,1);
            \coordinate (p11) at (0,2);
            \coordinate (p12) at (-2,4);
            \coordinate (p13) at (-3,5);
            \coordinate (p20) at (5,5);
            \coordinate (p21) at (4,6);
            \coordinate (p22) at (2,8);
            \coordinate (p23) at (1,9);

            \draw (p00)--(p01);
            \draw (p01)--(-1.5,1.5);
            \draw[dotted] (-1.75,1.75)--(-2.25,2.25);
            \draw (-2.5,2.5)--(p02);
            \draw (p02)--(p03);
            \draw (p10)--(p11);
            \draw (p11)--(-0.5,2.5);
            \draw[dotted] (-0.75,2.75)--(-1.25,3.25);
            \draw (-1.5,3.5)--(p12);
            \draw (p12)--(p13);
            \draw (p20)--(p21);
            \draw (p21)--(3.5,6.5);
            \draw[dotted] (3.25,6.75)--(2.75,7.25);
            \draw (2.5,7.5)--(p22);
            \draw (p22)--(p23);

            \draw (p00)--(p10);
            \draw (p01)--(p11);
            \draw (p02)--(p12);
            \draw (p03)--(p13);
            \draw (p10)--(1.5,1.5);
            \draw (p11)--(0.5,2.5);
            \draw (p12)--(-1.5,4.5);
            \draw (p13)--(-2.5,5.5);
            \draw[dotted] (1.75,1.75)--(4.25,4.25);
            \draw[dotted] (0.75,2.75)--(3.25,5.25);
            \draw[dotted] (-1.25,4.75)--(1.25,7.25);
            \draw[dotted] (-2.25,5.75)--(0.25,8.25);
            \draw (4.5,4.5)--(p20);
            \draw (3.5,5.5)--(p21);
            \draw (1.5,7.5)--(p22);
            \draw (0.5,8.5)--(p23);
  
            \draw [line width=0.05cm, fill=white] (p00) circle [radius=0.15];
            \draw [line width=0.05cm, fill=white] (p01) circle [radius=0.15];
            \draw [line width=0.05cm, fill=white] (p02) circle [radius=0.15];
            \draw [line width=0.05cm, fill=white] (p03) circle [radius=0.15];
            \draw [line width=0.05cm, fill=white] (p10) circle [radius=0.15];
            \draw [line width=0.05cm, fill=white] (p11) circle [radius=0.15];
            \draw [line width=0.05cm, fill=white] (p12) circle [radius=0.15];
            \draw [line width=0.05cm, fill=white] (p13) circle [radius=0.15];
            \draw [line width=0.05cm, fill=white] (p20) circle [radius=0.15];
            \draw [line width=0.05cm, fill=white] (p21) circle [radius=0.15];
            \draw [line width=0.05cm, fill=white] (p22) circle [radius=0.15];
            \draw [line width=0.05cm, fill=white] (p23) circle [radius=0.15];

            \draw [line width=0.02cm, decorate,decoration={brace,amplitude=15pt}](-4.25,4.25) -- (0.75,9.25) node[midway,xshift=-0.8cm,yshift=0.8cm] {\huge $l+1$}; 
            \draw [line width=0.02cm, decorate,decoration={brace,amplitude=15pt,mirror}](5.25,5.25) -- (1.25, 9.25) node[black,midway,xshift=0.8cm,yshift=0.8cm] {\huge $r+1$};
          \end{tikzpicture}
          }}
          \caption{$L(P)$}
          \label{fig:distributive lattice of 2-minors}
      \end{minipage}
  \end{figure}
  
  Finally, we consider the remaining case where $c_i$ and $d_j$ are comparable
  for some $i$ and $j$. Then, no additional poset ideal of
  $P$ appears compared to the previous case and some poset ideals appearing
  in the previous case vanish. In other words, let $Q$ be the poset described in the
  previous case. Then, we can check easily that $L(P)$ is a sublattice of $L(Q)$.
  Let $J \subset S_{L(Q)}$ be the initial ideal of $I_{L(Q)}$ with respect
  to a compatible monomial order, i.e.,
  \begin{equation*}
    J = (x_{\alpha}x_{\beta} : \alpha, \beta \in L(Q) \text{ are incomparable.}).
  \end{equation*}
  Similarly, the initial ideal of $I_{L(P)}$ with respect to a compatible monomial order
  $\preceq$ is
  \begin{equation*}
    \ini_{\preceq}(I_{L(P)}) = (x_{\alpha}x_{\beta} : \alpha, \beta \in L(P) \text{ are incomparable.}).
  \end{equation*}
  Since $L(P)$ is a sublattice of $L(Q)$, one has $\ini_{\preceq}(I_{L(P)}) = J \cap S_{L(P)}$.
  Thus, $\ini_{\preceq}(I_{L(P)})$ is a elimination ideal of the ideal of
  the Cartwright-Sturmfels ideal $J$.
  By Theorem~\ref{theorem:eliminated ideal of Cartwright-Sturmfels ideal},
  the initial ideal $\ini_{\preceq}(I_{L(P)})$ is a Cartwright-Sturmfels ideal, and
  $I_{L(P)}$ is also a Cartwright-Sturmfels ideal.
\end{proof}

\begin{example}\label{example:CS ideal of N-poset}
  We revisit the poset $P$
  and multigradings described in Example~\ref{example:multigradings of N-poset}.

  If $S_{L(P)}$ is multigraded by $C = \{2,3\}$, then $I_{L(P)}$ is not a
  Cartwright-Sturmfels ideal because $P \setminus C = \{1,4\}$ is not a chain.
  In particular, $x_{\{1,2\}}x_{\{2,4\}}$ is a monomial in the minimal generators of
  the initial ideal $\ini_{\preceq}(I_{L(P)})$ with respect to a
  compatible monomial order $\preceq$, and $\deg(x_{\{1,2\}}x_{\{2,4\}}) = 2\eb_1$.

  If $S_{L(P)}$ is multigraded by $C = \{2,4\}$, then $I_{L(P)}$ is a
  Cartwright-Sturmfels ideal because $P \setminus C = \{1,3\}$ is a chain.
  In particular, the initial ideal $\ini_{\preceq}(I_{L(P)})$ with respect to a compatible
  monomial order $\preceq$ coincides with a elimination ideal of the initial ideal
  of the ideal generated by the $2$-minors of the column-graded $3 \times 3$-matrix of variables
  with respect to a diagonal monomial order (Figure~\ref{fig:elimination ideal as a example of Cartwright-Sturmfels ideal}).
  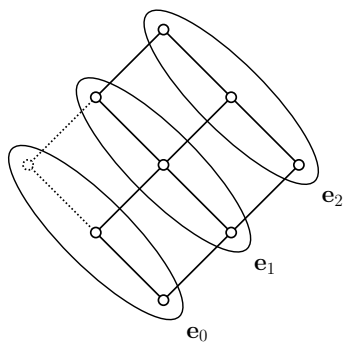
\begin{figure}
    {\scalebox{0.45}{
    \begin{tikzpicture}[line width=0.05cm]
      \coordinate (a1) at (0, 0);
      \coordinate (a2) at (-2, 2);
      \coordinate (a3) at (2, 2);
      \coordinate (a4) at (-4, 4);
      \coordinate (a5) at (0, 4);
      \coordinate (a6) at (4, 4);
      \coordinate (a7) at (-2, 6);
      \coordinate (a8) at (2, 6);
      \coordinate (a9) at (0, 8);
      
      \draw (a1) -- (a2);
      \draw (a1) -- (a3);
      \draw[dotted] (a2) -- (a4);
      \draw (a2) -- (a5);
      \draw (a3) -- (a5);
      \draw (a3) -- (a6);
      \draw[dotted] (a4) -- (a7);
      \draw (a5) -- (a7);
      \draw (a5) -- (a8);
      \draw (a6) -- (a8);
      \draw (a7) -- (a9);
      \draw (a8) -- (a9);
  
      \draw [line width=0.05cm, fill=white] (a1) circle [radius=0.15];
      \draw [line width=0.05cm, fill=white] (a2) circle [radius=0.15];
      \draw [line width=0.05cm, fill=white] (a3) circle [radius=0.15];
      \draw [dotted, line width=0.05cm, fill=white] (a4) circle [radius=0.15];
      \draw [line width=0.05cm, fill=white] (a5) circle [radius=0.15];
      \draw [line width=0.05cm, fill=white] (a6) circle [radius=0.15];
      \draw [line width=0.05cm, fill=white] (a7) circle [radius=0.15];
      \draw [line width=0.05cm, fill=white] (a8) circle [radius=0.15];
      \draw [line width=0.05cm, fill=white] (a9) circle [radius=0.15];

      \draw [very thick, rotate around={45:(-1,1)}] (a2) ellipse (1 and 3.5);
      \draw [very thick, rotate around={45:(0,5)}] (a5) ellipse (1 and 3.5);
      \draw [very thick, rotate around={45:(2,6)}] (a8) ellipse (1 and 3.5);

      \node at (1,-1) [font=\Huge] {$\eb_0$};
      \node at (3,1) [font=\Huge] {$\eb_1$};
      \node at (5,3) [font=\Huge] {$\eb_2$};
  \end{tikzpicture}
    }}
    \caption{Eliminating one variable}
    \label{fig:elimination ideal as a example of Cartwright-Sturmfels ideal}
  \end{figure}
\end{example}

\bibliography{references}
\bibliographystyle{amsplain}
\end{document}